\documentclass[a4paper,11pt]{amsart}

\usepackage{amsmath}
\usepackage{amssymb}
\usepackage{amsthm}
\usepackage{enumitem}
\usepackage{tikz-cd}
\usepackage{xcolor}

\usepackage{hyperref}

\newtheorem{thm}{Theorem}[section]
\newtheorem{lem}[thm]{Lemma}
\newtheorem{cor}[thm]{Corollary}
\theoremstyle{definition}

\theoremstyle{remark}
\newtheorem{rmk}[thm]{Remark}
\newtheorem{nota}[thm]{Notation}

\numberwithin{equation}{section}

\makeatletter
\g@addto@macro\bfseries{\boldmath}
\makeatother

\newcommand{\Aut}{\mathrm{Aut}}   
\newcommand{\End}{\mathrm{End}}   
\newcommand{\eps}{\varepsilon}
\newcommand{\Fix}{\mathrm{Fix}}   
\newcommand{\fix}{\mathrm{fix}}   
\newcommand{\coin}{\mathrm{coin}} 
\newcommand{\g}{\mathfrak{g}}     
\renewcommand{\iff}{\enskip\Leftrightarrow\enskip}
\renewcommand{\mod}{\;\mathrm{mod}\,}
\newcommand{\orb}{\backslash}     
\newcommand{\R}{\mathcal{R}}      
\newcommand{\RR}{\mathbb{R}}      
\renewcommand{\S}{\Sigma}         
\newcommand{\ZZ}{\mathbb{Z}}      

\begin{document}

\title[Nielsen numbers of affine $n$-valued maps]{An averaging formula for Nielsen numbers of affine $n$\hspace{1pt}-\hspace{-0.5pt}valued maps on infra-nilmanifolds}
\author{Karel Dekimpe and Lore De Weerdt}
\thanks{Research supported by Methusalem grant METH/21/03 -- long term structural funding of the Flemish Government.}
\address{KU Leuven Campus Kulak Kortrijk, 8500 Kortrijk, Belgium}
\email{Karel.Dekimpe@kuleuven.be}
\email{Lore.DeWeerdt@kuleuven.be}
\begin{abstract}
In \cite{kimleelee,leelee}, the authors developed a nice formula to compute the Nielsen number of a self-map on an infra-nilmanifold. For the case of nilmanifolds this formula was extended to $n$-valued maps in \cite{charlotte2}. In this paper, we extend these results further and establish the averaging formula to compute the Nielsen number of any $n$-valued affine map on an infra-nilmanifold.
\end{abstract}

\maketitle

\section*{Introduction}

Let $M$ be a closed manifold (or more generally a compact polyhedron) and let $f:M\to M$ be a self-map of $M$. In Nielsen-Reidemeister fixed point theory one attaches to each such map $f$ a number $N(f)$, the Nielsen number of $f$, which is a (sharp) lower bound for the number of fixed points of any map $g$ which is homotopic to $f$. In this paper we are concerned with infra-nilmanifolds. 
It is known that any self-map on an infra-nilmanifold is homotopic to a so-called affine map (\cite{kblee}). These maps are rather algebraic in nature and one has developed an averaging formula for computing their Nielsen number (\cite{kimleelee,leelee}). Since the Nielsen number of a map is invariant under a homotopy, this formula allows us to compute the Nielsen number of any self-map on an infra-nilmanifold.
For $n$-valued maps  we also have the concept of affine maps (see \cite{lore} and/or section \ref{sec:affine}), but there the situation is more complicated. It is no longer true that any $n$-valued map on an infra-nilmanifold is homotopic to an affine $n$-valued map (\cite{lore}), but nevertheless, these affine $n$-valued maps still form a large and interesting class of $n$-valued maps on infra-nilmanifolds. In this paper, we obtain a natural generalisation of the formula from \cite{leelee} for computing the Nielsen number of any affine $n$-valued map on an infra-nilmanifold.

In the first section we recall the basic Nielsen theory for $n$-valued maps and describe a general averaging formula \eqref{eq:N(f)} for computing the Nielsen number of an $n$-valued map. In the second section, we recall the necessary concepts concerning infra-nilmanifolds and affine $n$-valued maps. In the third section we establish a decomposition of the fixed point set of an $n$-valued map on an infra-nilmanifold, which allows us to rewrite the general averaging formula into a more algebraic one \eqref{eq:decomp-N}. In the fourth section we further examine this formula in the case of affine $n$-valued maps, to finally arrive at the main result of this paper (Theorem~\ref{mainthm}), which is an averaging formula generalising in a natural way the formula from \cite{leelee}. 
Finally in the last section we illustrate our main result by means of an affine 2-valued map on the Klein bottle. Moreover, the example we provide is special in the sense that this map  does not lift to a 2-valued map on any torus (nilmanifold) that finitely covers the Klein bottle.

\section{Nielsen theory of $n$-valued maps}\label{sec:intro}

Let $X$ be a connected compact manifold with fundamental group $\pi$ and universal cover $p:\tilde{X}\to X$. We will view $\pi$ as being the group of covering transformations of $p$. 

An \emph{n-valued map} $f:X\multimap X$ is a continuous (i.e. upper and lower semi-continuous) set-valued function, sending each point of $X$ to a set of $n$ distinct points in $X$. Equivalently, by the work of
Brown and Gon\c{c}alves \cite{browngoncalves}, the map $f$ can be viewed as a continuous single-valued function from $X$ to the \emph{unordered configuration space} \[
D_n(X)=\{\{x_1,\ldots,x_n\}\subseteq X\mid x_i\neq x_j \text{ if } i\neq j \},
\]
i.e. the quotient of the space $F_n(X)=\{(x_1,\ldots,x_n)\in X^n\mid x_i\neq x_j \text{ if } i\neq j \}$ under the action of the permutation group $\S_n$. The fixed point set of an $n$-valued map $f$ is $\Fix(f)=\{x\in X \mid x\in f(x) \}$.

In what follows we recall some of the results of Brown et al.\ \cite{charlotte}, where the study of fixed points of $f$ was done by considering fixed points of \emph{lifts} of $f$ to the \emph{orbit configuration space} \[
F_n(\tilde{X},\pi)=\{(\tilde{x}_1,\ldots,\tilde{x}_n)\in \tilde{X}^n \mid p(\tilde{x}_i)\neq p(\tilde{x}_j) \text{ if } i\neq j \},
\]
which is a covering space of $D_n(X)$ with covering map \[
p^n:F_n(\tilde{X},\pi)\to D_n(X):(\tilde{x}_1,\ldots,\tilde{x}_n)\mapsto \{p(\tilde{x}_1),\ldots,p(\tilde{x}_n)\}.
\]
The corresponding covering group is $\pi^n\rtimes \S_n$, which acts on $F_n(\tilde{X},\pi)$ by \[
(\gamma_1,\ldots,\gamma_n;\sigma)(\tilde{x}_1,\ldots,\tilde{x}_n)=(\gamma_1\tilde{x}_{\sigma^{-1}(1)},\ldots,\gamma_n\tilde{x}_{\sigma^{-1}(n)}).
\]
A lift of $f$ can then be defined as a map $\tilde{f}:\tilde{X}\to F_n(\tilde{X},\pi)$ such that \[
\begin{tikzcd}
\tilde{X} \ar[r,"\tilde{f}"] \ar[d,"p"'] & F_n(\tilde{X},\pi) \ar[d,"p^n"] \\
X \ar[r,"f"] & D_n(X)
\end{tikzcd}
\]
commutes. As $F_n(\tilde{X},\pi)$ is a subspace of $\tilde{X}^n$, such a lift $\tilde{f}:\tilde{X}\to F_n(\tilde{X},\pi)$ splits into $n$ maps $\tilde{f}_1,\ldots,\tilde{f}_n:\tilde{X}\to\tilde{X}$, called the \emph{lift-factors} of $\tilde{f}$. 

A lift $\tilde{f}=(\tilde{f}_1,\ldots,\tilde{f}_n)$ induces a morphism
$(\phi_1,\ldots,\phi_n;\sigma):\pi\to \pi^n\rtimes \S_n$ of the covering groups so that for all $\gamma\in \pi$, \[
(\tilde{f}_1\gamma,\ldots,\tilde{f}_n\gamma)=(\phi_1(\gamma)\tilde{f}_{\sigma_\gamma^{-1}(1)},\ldots,\phi_n(\gamma)\tilde{f}_{\sigma_\gamma^{-1}(n)}).
\]
Using this morphism, one can first subdivide $\{1,\ldots,n\}$ into \emph{$\sigma$-classes}, the equivalence classes for the relation given by \[
i\sim j \iff \exists \gamma\in\pi:\sigma_{\gamma}(i)=j.
\]
For all $i$, denote by $S_i=\{\gamma\in\pi\mid \sigma_{\gamma}(i)=i \}$ the stabilizer of $i$ for this relation, which is a finite index subgroup of $\pi$. The restrictions of the maps $\phi_i$ to $S_i$ are group morphisms, using which one can define an equivalence relation on $\pi$ by \[
\alpha\sim\beta \iff \exists \gamma\in S_i:\alpha=\gamma\beta\phi_i(\gamma)^{-1}.
\]
The set of equivalence classes for this relation is denoted $\R[\phi_i]$, and its number of elements is the \emph{Reidemeister number} $R(\phi_i)$.

\begin{rmk}
In general, given two group morphisms $\varphi,\psi:\Pi_1\to\Pi_2$, one can consider the equivalence relation on $\Pi_2$ given by \[
\alpha\sim \beta \iff \exists \gamma\in \Pi_1:\alpha=\psi(\gamma)\beta\varphi(\gamma)^{-1}.
\]
The equivalence classes are called the \emph{Reidemeister coincidence classes} of $\varphi$ and $\psi$, and denoted $\R[\varphi,\psi]$. In case $\Pi_1$ is a subgroup of $\Pi_2$ and $\psi$ is the inclusion $\iota:\Pi_1\to\Pi_2$, we abbreviate $\R[\varphi,\iota]=\R[\varphi]$.
\end{rmk}

The $\sigma$-classes and Reidemeister classes of $\phi_i$ partition the fixed point set of $f$ into a disjoint union of \emph{fixed point classes}: if $i_1,\ldots,i_r$ are representatives of the $\sigma$-classes, then \[
\Fix(f)=\bigsqcup_{\ell=1}^r \bigsqcup_{[\alpha]\in \R[\phi_{i_\ell}]} p(\Fix(\alpha\tilde{f}_{i_\ell})).
\]
To each fixed point class one can associate an \emph{index}, and the fixed point classes with non-zero index are called \emph{essential}. The number of essential fixed point classes is the \emph{Nielsen number} of $f$, which is a lower bound for the number of fixed points among all maps homotopic to $f$: \[
N(f)\leq \text{min}\{\# \Fix(g) \mid g\sim f \}.
\] 
By the above, the Nielsen number can be written as \[
N(f)=\sum_{\ell=1}^r\sum_{[\alpha]\in \R[\phi_{i_\ell}]}\eps_{\alpha,i_\ell}
\]
where \[
\varepsilon_{\alpha,i}=\left\{
\begin{array}{ll}
1 & \text{if $p(\Fix(\alpha\tilde{f}_i))$ is an essential fixed point class of $f$,} \\
0 & \text{otherwise.}
\end{array}
\right.
\]
Since the representatives for the $\sigma$-classes were chosen arbitrarily, and choosing another representative $i_\ell$ for some $\ell$ must give the same result for $N(f)$, the number \[
\sum_{[\alpha]\in \R[\phi_{i_\ell}]}\eps_{\alpha,i_\ell}
\]
does not depend on the chosen representative $i_\ell$. Therefore the formula above can also be written as \begin{equation}\label{eq:N(f)}
N(f)=\sum_{i=1}^n\frac{1}{[\pi:S_i]}\sum_{[\alpha]\in \R[\phi_i]}\eps_{\alpha,i}.
\end{equation}

\section{Infra-nilmanifolds and affine maps}\label{sec:affine}

The maps we consider in this paper are \emph{affine} $n$-valued maps on \emph{infra-nilmanifolds}. An infra-nilmanifold, as defined e.g.\ in \cite{user-guide-inm}, is the quotient of a connected and simply connected nilpotent Lie group $G$ under the action of an almost-Bieberbach group $\pi\subseteq G\rtimes \Aut(G)$, i.e.\ a torsion free cocompact discrete subgroup of $G\rtimes C$, with $C$ a maximal compact subgroup of $\Aut(G)$. Such an almost-Bieberbach group acts properly discontinuously on $G$ and hence $\pi$ is (isomorphic to) the fundamental group of the infra-nilmanifold $\pi\orb G$. As before, we will view $\pi$ as the covering group of the universal covering $G \to \pi\orb G$.

In the special case where this group $\pi$ is entirely contained in $G$, the quotient space is called a nilmanifold. Every infra-nilmanifold $\pi\orb G$ is finitely covered by a nilmanifold $N\orb G$: it follows from the generalised Bieberbach Theorems that $\pi \cap G$ is of finite index in $\pi$, so one can take $N=\pi\cap G$ or any other subgroup 
$N \subseteq \pi\cap G$ which is of finite index in $\pi$. Such a group $N$ is a finitely generated torsion free nilpotent group and is a lattice of the Lie group $G$.

Nilmanifolds have been studied in detail by Mal'cev; see \cite{malcev}. We briefly recall some of his results.

For a nilpotent group $G$ of class $c$, we denote the lower central series by \[
G=\gamma_1(G)\geq \ldots \geq \gamma_c(G) \geq \gamma_{c+1}(G)=1.
\]
For a finitely generated torsion free nilpotent group $N$ of class $c$ we consider another central series whose terms are the \emph{isolators} of the lower central series, \[
N_i=\sqrt[N]{\gamma_i(N)}=\{g\in N \mid \exists k>0:g^k\in \gamma_i(N) \}.
\]
This series has the useful property that for all $i$, the quotient $N_i/N_{i+1}$ is free abelian of finite rank, say $N_i/N_{i+1}\cong \ZZ^{k_i}$. Moreover, for these numbers $k_i$, the factors of the lower central series of $G$ satisfy $\gamma_i(G)/\gamma_{i+1}(G)\cong \RR^{k_i}$. Morphisms of (lattices in) $G$ induce morphisms on these quotients, which can be represented by $k_i\times k_i$ real (or integer) matrices.

If $N,N'$ are lattices in $G$, any morphism $\phi:N\to N'$ extends uniquely to a Lie group morphism $G\to G$. The latter induces a morphism on the Lie algebra $\g$ of $G$. From now on we will refer to this as the Lie algebra morphism induced by $\phi:N\to N'$, and write $\phi_*:\g\to\g$, without explicitly mentioning the extension of $\phi$ to $G$.

\smallskip

Another useful property of connected and simply connected nilpotent Lie groups $G$ is that the exponential map $\exp:\g\to G$ is a diffeomorphism \cite[Theorem 1.2.1]{corwingreenleaf}. In particular, it is invertible, and we denote the inverse map by $\log:G\to \g$. This allows us to define the continuous function \[
G\times \RR\to G: (g,t) \mapsto g^t=\exp(t\log(g))
\]
which we will use later on.

On an infra-nilmanifold we can define the notion of an \emph{affine} $n$-valued map. A map $f:\pi\orb G\to D_n(\pi\orb G)$ is affine if it lifts to a map of the form \[
\tilde{f}:G\to F_n(G,\pi):x\mapsto (g_1\varphi_1(x),\ldots,g_n\varphi_n(x))
\]
with $g_1,\ldots,g_n\in G$ and $\varphi_1,\ldots,\varphi_n\in \End(G)$. We remark that this condition is independent of the chosen lift $\tilde{f}$.
The corresponding maps $\phi_i$ and $\sigma$ from section \ref{sec:intro} satisfy \begin{equation}\label{eq:psi-aff}
g_i\varphi_i(\gamma x)=\phi_i(\gamma)g_{\sigma_\gamma^{-1}(i)}\varphi_{\sigma_\gamma^{-1}(i)}(x)
\end{equation}
for all $\gamma\in \pi$, $x\in G$ and $i\in \{1,\ldots,n\}$.

Nielsen numbers of affine $n$-valued maps on nilmanifolds have been studied in \cite{charlotte2}, where the following formula is proven:


\begin{thm}[{\cite[Theorem 6.13]{charlotte2}}]
Let $N\orb G$ be a nilmanifold. The Nielsen number of an affine map $f:N\orb G\to D_n(N\orb G)$ with lift $\tilde{f}:G\to F_n(G,N):x\mapsto (g_1\varphi_1(x),\ldots,g_n\varphi_n(x))$ is given by \[
N(f)=\sum_{i=1}^n |\det(I-(\varphi_i)_*)|.
\]
\end{thm}

The goal of this paper is to generalise this formula to infra-nilmanifolds. In the single-valued case, this extension from nilmanifolds to infra-nilmani\-folds was done by Kim, Lee and Lee in \cite{kimleelee}. They proved that any map $f$ on an infra-nilmanifold $\pi\orb G$ can be lifted to a nilmanifold $N\orb G$ that finitely covers $\pi\orb G$, and they expressed the Nielsen number of $f$ as the average of the Nielsen numbers of these lifts.

We cannot follow the same approach in the $n$-valued case, since $n$-valued maps in general no longer admit lifts to a nilmanifold. However, we can decompose the formula for the Nielsen number in a similar way as in \cite{kimleelee}, and then plug in the knowledge about affine maps on nilmanifolds established in \cite{charlotte2} to obtain the desired averaging formula without involving actual lifts to the nilmanifold.

\section{Decomposition of the fixed point set}\label{sec:decomp}

Let $\pi\orb G$ be an infra-nilmanifold, finitely covered by a nilmanifold $N\orb G$. For a general map $f:\pi\orb G\to D_n(\pi\orb G)$, we will decompose the sum (\ref{eq:N(f)}) into a sum over Reidemeister classes in $N$ and in $\pi/N$, generalising the approach in \cite{kimleelee}.

For all $i$, consider the group $S'_i$ generated by $\{\gamma^{[\pi:N]} \mid \gamma\in S_i \}$. This is a normal subgroup of $S_i$, and the quotient group $S_i/S'_i$ is periodic, since the order of every element divides $[\pi:N]$. Also note that $\pi$ (and therefore also $S_i/S'_i$) is polycyclic-by-finite, since its finite index normal subgroup $N$ is finitely generated and nilpotent. Since periodic polycyclic-by-finite groups are automatically finite, it follows that $[S_i:S'_i]$ is finite. Consequently, also \[
[N:S'_i]=\frac{[\pi:S_i][S_i:S'_i]}{[\pi:N]}
\]
is finite. Thus, $S_i'$ is a finite index normal subgroup of $S_i$ and a finite index subgroup of $N$.

Furthermore, note that the image of $S'_i$ under any morphism $S_i\to \pi$ is contained in $N$. Hence the inclusion map $\iota:S_i\to \pi$ and the morphism $\phi_i:S_i\to \pi$ induce morphisms \begin{align*}
\iota':S_i'\to N&:\gamma\mapsto \gamma & \bar{\iota}:S_i/S_i'\to \pi/N&:\gamma S_i'\mapsto \gamma N \\ 
\phi'_i:S_i'\to N&:\gamma\mapsto \phi_i(\gamma) & \bar{\phi}_i:S_i/S_i'\to \pi/N&:\gamma S_i'\mapsto \phi_i(\gamma)N.
\end{align*}
We will decompose the set of Reidemeister classes of $\phi_i$ using Reidemeister classes associated to $\bar{\phi}_i$ and $\phi_i'$.

First, note that \[
\R[\phi_i]=\bigcup_{\bar{\alpha}\in \pi/N}\,\bigcup_{g\in N}\,[g\alpha].
\]
If $[g\alpha]=[g'\alpha']$ in $\R[\phi_i]$, then certainly $[\bar{\alpha}]=[\bar{\alpha}']$ in $\R[\bar{\phi}_i,\bar{\iota}]$, i.e. \[
\exists \bar{\gamma}\in S_i/S'_i: \bar{\alpha}=\bar{\iota}(\bar{\gamma})\bar{\alpha}'\bar{\phi}_i(\bar{\gamma})^{-1}.
\]
Conversely, note that if $[\bar{\alpha}]=[\bar{\alpha}']$ in $\R[\bar{\phi}_i,\bar{\iota}]$, then there is a $g\in N$ so that $[g\alpha]=[\alpha']$ in $\R[\phi_i]$.

For fixed $\alpha$, we have $[g\alpha]=[g'\alpha]$ in $\R[\phi_i]$ if and only if $[g]=[g']$ in $\R[\tau_\alpha\phi_i]$, with $\tau_\alpha:\pi\to\pi:\beta\mapsto \alpha\beta\alpha^{-1}$; i.e. \[
\exists \gamma\in S_i:g=\gamma g'\alpha\phi_i(\gamma)^{-1}\alpha^{-1}.
\]
Thus, the above are disjoint unions when taken over these respective Reidemeister classes: \[
\R[\phi_i]=\bigsqcup_{[\bar{\alpha}]\in \R[\bar{\phi}_i,\bar{\iota}]}\bigsqcup_{[g]\in \R[\tau_\alpha\phi_i]}\,[g\alpha].
\]

Rather than in $\R[\tau_\alpha\phi_i]$, we would like to view the classes $[g]$ with $g\in N$ as elements of $\R[\tau_\alpha\phi'_i]$, where \[
g\sim g'\iff \exists \gamma\in S'_i:g=\gamma g'\alpha\phi'_i(\gamma)^{-1}\alpha^{-1}.
\]
Since $S_i'\subseteq S_i$, it is clear that $[g]=[g']$ in $\R[\tau_\alpha\phi'_i]$ implies $[g]=[g']$ in $\R[\tau_\alpha\phi_i]$, but conversely there can be distinct elements in $\R[\tau_\alpha\phi'_i]$ that are equivalent in $\R[\tau_\alpha\phi_i]$.
The following lemma allows us to describe how many different classes in $\R[\tau_\alpha\phi'_i]$ give rise to the same class in $\R[\tau_\alpha\phi_i]$.
In this lemma, we use $\coin(\varphi,\psi)=\{\gamma\in \Pi_1\mid \varphi(\gamma)=\psi(\gamma) \}$ to denote the group of coincidence points of two morphisms $\varphi,\psi:\Pi_1\to\Pi_2$.

\begin{lem}[{\cite[{Lemma 3.1.4}]{penninckx}}]
Given a commutative diagram of groups with exact rows \[
\begin{tikzcd}
1 \ar[r] & \Gamma_1 \ar[r] \ar[d,"\varphi'"',"\psi'"] & \Pi_1 \ar[r] \ar[d,"\varphi"',"\psi"] & \Pi_1/\Gamma_1 \ar[r] \ar[d,"\bar{\varphi}"',"\bar{\psi}"] & 1 \\
1 \ar[r] & \Gamma_2 \ar[r] & \Pi_2 \ar[r] & \Pi_2/\Gamma_2 \ar[r] & 1
\end{tikzcd}
\]
such that the quotient groups $\Pi_i/\Gamma_i$ are finite. For all $\alpha\in \Pi_2$ and $g\in \Gamma_2$, the number of Reidemeister coincidence classes $[g']\in \R[\tau_\alpha\varphi',\psi']$ such that $[g]=[g']$ in $\R[\tau_\alpha\varphi,\psi]$ is equal to \[
[\coin(\tau_{\bar{\alpha}}\bar{\varphi},\bar{\psi}):\overline{\coin(\tau_{g\alpha}\varphi,\psi)}]
\]
where $\overline{\coin(\tau_{g\alpha}\varphi,\psi)}$ denotes the image of $\coin(\tau_{g\alpha}\varphi,\psi)$ in $\Pi_1/\Gamma_1$.
\end{lem}

\noindent
Applying this to the diagram \[
\begin{tikzcd}
1 \ar[r] & S'_i \ar[r] \ar[d,"\phi'_i"',"\iota'"] & S_i \ar[r] \ar[d,"\phi_i"',"\iota"] & S_i/S'_i \ar[r] \ar[d,"\bar{\phi}_i"',"\bar{\iota}"] & 1 \\
1 \ar[r] & N \ar[r] & \pi \ar[r] & \pi/N \ar[r] & 1
\end{tikzcd}
\]
learns that for a given $g\in N$, the number of classes $[g']\in \R[\tau_\alpha\phi'_i]$ such that $[g]=[g']$ in $\R[\tau_\alpha\phi_i]$ is equal to \[
[\coin(\tau_{\bar{\alpha}}\bar{\phi}_i,\bar{\iota}):\overline{\fix(\tau_{g\alpha}\phi_i)}]=\frac{|\coin(\tau_{\bar{\alpha}}\bar{\phi}_i,\bar{\iota})|}{|\overline{\fix(\tau_{g\alpha}\phi_i)}|},
\]
where $\overline{\fix(\tau_{g\alpha}\phi_i)}$ denotes the image of $\fix(\tau_{g\alpha}\phi_i)=\coin(\tau_{g\alpha}\phi_i,\iota)$ in $S_i/S_i'$.
Therefore the sum (\ref{eq:N(f)}) decomposes as \[
N(f)=\sum_{i=1}^n\frac{1}{[\pi:S_i]}\sum_{[\bar{\alpha}]\in \R[\bar{\phi}_i,\bar{\iota}]}\sum_{[g]\in \R[\tau_\alpha\phi'_i]}\frac{|\overline{\fix(\tau_{g\alpha}\phi_i)}|}{|\coin(\tau_{\bar{\alpha}}\bar{\phi}_i,\bar{\iota})|}\,\varepsilon_{g\alpha,i}.
\]

To further simplify this formula, consider the action of $S_i/S'_i$ on $\pi/N$ given by \[
\bar{\gamma}\cdot \bar{\alpha}=\bar{\iota}(\bar{\gamma})\bar{\alpha}\bar{\phi}_i(\bar{\gamma})^{-1}.
\]
The orbit of an element $\bar{\alpha}$ for this action is its Reidemeister class $[\bar{\alpha}]\in \R[\bar{\phi}_i,\bar{\iota}]$. The stabilizer of $\bar{\alpha}$ is the set \[
\{\bar{\gamma}\in S_i/S'_i \mid \bar{\iota}(\bar{\gamma})\bar{\alpha}\bar{\phi}_i(\bar{\gamma})^{-1}=\bar{\alpha} \}=\coin(\tau_{\bar{\alpha}}\bar{\phi}_i,\bar{\iota}).
\]
Therefore we have \[
|[\bar{\alpha}]|\cdot |\coin(\tau_{\bar{\alpha}}\bar{\phi}_i,\bar{\iota})|=|S_i/S'_i|=[S_i:S'_i].
\]
This reduces the expression for $N(f)$ to \begin{align}
N(f)&=\sum_{i=1}^n\frac{1}{[\pi:S_i]}\sum_{\bar{\alpha}\in \pi/N}\frac{1}{|[\bar{\alpha}]|}\sum_{[g]\in \R[\tau_\alpha\phi'_i]}\frac{|\overline{\fix(\tau_{g\alpha}\phi_i)}|}{|\coin(\tau_{\bar{\alpha}}\bar{\phi}_i,\bar{\iota})|}\,\varepsilon_{g\alpha,i} \notag \\&=\sum_{i=1}^n\frac{1}{[\pi:S'_i]}\sum_{\bar{\alpha}\in \pi/N}\sum_{[g]\in \R[\tau_\alpha\phi'_i]}|\overline{\fix(\tau_{g\alpha}\phi_i)}|\,\varepsilon_{g\alpha,i}. \label{eq:decomp-N}
\end{align}

\section{Averaging formula for affine maps}

Now we apply this to the special case where $f$ is an affine map with lift $\tilde{f}:G\to F_n(G,\pi):x\mapsto (g_1\varphi_1(x),\ldots,g_n\varphi_n(x))$. In this case, the formula above can be reduced to an expression purely in terms of the morphisms $\varphi_i$.

Recall that, for a morphism $\phi':S'\to N$ between lattices in a connected and simply connected nilpotent Lie group $G$, we use $(\phi')_*:\g\to\g$ to denote the Lie algebra morphism induced by the unique extension of $\phi'$ to $G\to G$.

\begin{lem}\label{lem:fix}
Suppose $\pi\orb G$ is an infra-nilmanifold finitely covered by a nilmanifold $N\orb G$, and $S$ is a finite index subgroup of $\pi$. Let $S'$ be the subgroup generated by $\{\gamma^{[\pi:N]} \mid \gamma\in S \}$.
Let $\phi:S\to \pi$ be a morphism and $\phi':S'\to N$ its restriction to $S'$. If $\det(I-(\phi')_*)\neq 0$, then $\fix(\phi)=1$.
\end{lem}

\begin{proof}
We first prove that $\fix(\phi')=1$. To this end, let $\exp:\g\to G$ denote the exponential map of $G$, and $\log:G\to \g$ its inverse.

Suppose $\gamma\in S'$ is a fixed point of $\phi'$. Then $X=\log(\gamma)$ is a fixed point of $(\phi')_*$, i.e. $(I-(\phi')_*)(X)=0$. Since $\det(I-(\phi')_*)\neq 0$, the only element $X$ satisfying this equation is $0$. It follows that $\gamma=\exp(0)=1$.

Now, suppose $\gamma\in S$ is a fixed point of $\phi$, and denote $k=[\pi:N]$. Then $\phi(\gamma^k)=\phi(\gamma)^k=\gamma^k$, so $\gamma^k\in S'$ is a fixed point of $\phi'$. By the above, we have $\gamma^k=1$. Since $S$ is torsion free, it follows that $\gamma=1$.
\end{proof}

We will now show that for an affine map, the number $\eps_{g\alpha,i}$ in (\ref{eq:decomp-N}) is independent of $g\in N$, using the results from \cite{charlotte2}. We distinguish two cases. In both cases we will need the following lemma.

\begin{lem}[{\cite[Lemma 6.4]{charlotte2}}]\label{lem:no-ev1}
Let $G$ be a connected and simply connected nilpotent Lie group, and $\varphi\in\End(G)$. If $\varphi_*$ has no eigenvalue $1$, the map $G\to G:x\mapsto \varphi(x)x^{-1}$ is bijective.
\end{lem}

The first case is the following. 

\begin{nota}
For $\alpha\in\pi\subseteq G\rtimes \Aut(G)$ we will write $\alpha=(a_\alpha,A_\alpha)$ with $a_\alpha\in G$ and $A_\alpha\in \Aut(G)$.
\end{nota}

\begin{cor}\label{cor:eps1}
Let $f:\pi\orb G\to D_n(\pi\orb G)$ be an affine $n$-valued map with lift $\tilde{f}:x\mapsto (g_1\varphi_1(x),\ldots,g_n\varphi_n(x))$ and corresponding numbers $\varepsilon_{g\alpha,i}$. Fix $i$ and $\alpha\in \pi$. If $\det(I-(A_\alpha\varphi_i)_*)\neq 0$, then $\varepsilon_{g\alpha,i}=1$ for all $g\in N$.
\end{cor}

\begin{proof}
Pick $g\in N$. Note that for all $x\in G$, \[
g\alpha\tilde{f}_i(x)=g\alpha g_i\varphi_i(x)=ga_\alpha A_\alpha(g_i)A_\alpha\varphi_i(x).
\]
So if we denote $ga_\alpha A_\alpha(g_i)=g'$, then $g\alpha\tilde{f}_i=g'A_\alpha\varphi_i$.

By Lemma \ref{lem:no-ev1}, the map $G\to G: x\mapsto A_\alpha\varphi_i(x)x^{-1}$ is a bijection. Hence there is a unique $x\in G$ such that $A_\alpha\varphi_i(x)x^{-1}=(g')^{-1}$, i.e. such that $g'A_\alpha\varphi_i(x)=x$. That is, the map $g\alpha\tilde{f}_i=g'A_\alpha\varphi_i$ has a unique fixed point. The index of the fixed point class $p(\Fix(g\alpha\tilde{f}_i))$ is then equal to the index of this unique fixed point of $g'A_\alpha\varphi_i$.

In \cite[Lemma 6.10]{charlotte2} it is proven that, if $\varphi:G\to G$ is a morphism so that $G\to G:x\mapsto \varphi(x)x^{-1}$ is a bijection, then for any $h\in G$, the index of $G\to G:x\mapsto h\varphi(x)$ at its unique fixed point is $\pm 1$. Applying this to $\varphi=A_\alpha\varphi_i$ and $h=g'$ gives the desired result.
\end{proof}

For the second case, we can use the following lemma, which is a special case of \cite[Lemma 6.6]{charlotte2}. We give a more polished version of the proof adapted to our specific needs.

\begin{lem}\label{lem:ev1}
Let $G$ be a connected and simply connected nilpotent Lie group, and $\varphi\in\End(G)$. If $\varphi_*$ has eigenvalue $1$, then for all $g\in G$ there is a $g_0\in G$ such that $g\varphi(x)g_0\neq x$ for all $x\in G$.
\end{lem}

\begin{proof}
We prove the claim by induction on the nilpotency class $c$ of $G$.

If $c=1$, then $G=\RR^k$ for some $k$ and $\varphi=\varphi_*$ is a linear map represented by a matrix which has eigenvalue $1$. We can choose a basis $\{v_1,\ldots,v_k\}$ of $\RR^k$ with respect to which the matrix of $\varphi$ is of the form \[
\begin{bmatrix}
1 & 0 & \cdots & 0 \\
* & * & \cdots & * \\
\vdots & \vdots & \ddots & \vdots \\
* & * & \cdots & *
\end{bmatrix}.
\]
That is, if we write $V=\langle v_2,\ldots,v_k\rangle$, then $\varphi(\mu v_1+v)\equiv \mu v_1\mod V$ for all $\mu\in \RR$ and $v\in V$.

Take $\lambda\in \RR$ such that $g\equiv \lambda v_1\mod V$. We claim that $g_0=(1-\lambda)v_1$ does the job. Indeed, suppose $x\in \RR^k$ satisfies $g+\varphi(x)+g_0=x$. If we write $x=\mu v_1+v$ with $v\in V$, this reduces to \[
\lambda v_1+\mu v_1+(1-\lambda)v_1 \equiv \mu v_1\mod V,
\]
or equivalentely, $v_1\in V$, which is a contradiction.

\medskip

Now suppose $c>1$ and the result holds for all groups of smaller nilpotency class. For all $i=1,\ldots,c$, let $M_i$ be the matrix of the morphism induced by $\varphi$ on $\gamma_i(G)/\gamma_{i+1}(G)$. We distinguish two cases: either one of the matrices $M_1,\ldots,M_{c-1}$ has eigenvalue $1$, or $M_c$ is the only matrix with eigenvalue $1$.

In the first case, we can apply the induction hypothesis to the map $G/\gamma_c(G)\to G/\gamma_c(G):\bar{x}\mapsto \overline{\varphi(x)}$ (and the element $\bar{g}$ for $g$) to obtain a $g_0\in G$ such that $\overline{g\varphi(x)g_0}\neq\bar{x}$ in $G/\gamma_c(G)$, so in particular $g\varphi(x)g_0\neq x$ in $G$, for all $x\in G$.

In the second case, we can apply the induction hypothesis to the restriction of $\varphi$ to $\gamma_c(G)$ (and the identity element for $g$) to obtain a $g_0\in \gamma_c(G)$ such that $\varphi(x)g_0\neq x$ for all $x\in \gamma_c(G)$. On the other hand, we can apply Lemma \ref{lem:no-ev1} to the map $G/\gamma_c(G)\to G/\gamma_c(G):\bar{x}\mapsto \overline{\varphi(x)}$, which has no eigenvalue $1$. It follows that there is a unique $\bar{x}_0\in G/\gamma_c(G)$ such that $\overline{x_0\varphi(x_0)^{-1}}=\bar{g}$. Choose a representative $x_0\in G$ for $\bar{x}_0$, and take $z\in \gamma_c(G)$ such that $g\varphi(x_0)=x_0z$. We claim that the element $g_0z^{-1}$ does the job. Indeed, suppose $x\in G$ satisfies $g\varphi(x)g_0z^{-1}=x$. By uniqueness of $\bar{x}_0$, this element $x$ must be of the form $x_0z'$ for some $z'\in \gamma_c(G)$. Then we have \begin{align*}
g\varphi(x_0z')g_0z^{-1}=x_0z' &\iff g\varphi(x_0)\varphi(z')g_0z^{-1}=x_0z' \\
&\iff x_0z\varphi(z')g_0z^{-1}=x_0z' \\
&\iff \varphi(z')g_0=z',
\end{align*}
which is impossible for $z'\in \gamma_c(G)$, by definition of $g_0$.
\end{proof}

\begin{cor}\label{cor:eps0}
Let $f:\pi\orb G\to D_n(\pi\orb G)$ be an affine $n$-valued map with lift $\tilde{f}:x\mapsto (g_1\varphi_1(x),\ldots,g_n\varphi_n(x))$ and corresponding numbers $\varepsilon_{g\alpha,i}$. Fix $i$ and $\alpha\in \pi$. If $\det(I-(A_\alpha\varphi_i)_*)= 0$, then $\varepsilon_{g\alpha,i}=0$ for all $g\in N$.
\end{cor}

\begin{proof}
Pick $g\in N$ and define $g'$ as in Corollary \ref{cor:eps1}. By Lemma \ref{lem:ev1}, there exits a $g_0\in G$ such that $g'A_\alpha\varphi_i(x)g_0\neq x$ for all $x\in G$.
Then we have a homotopy \[
\pi\orb G\times [0,1]\to D_n(\pi\orb G):(p(x),t)\mapsto \{p(g\alpha\tilde{f}_1(x)g_0^t),\ldots,p(g\alpha\tilde{f}_n(x)g_0^t) \}.
\]
Under this homotopy, the fixed point class $p(\Fix(g\alpha \tilde{f}_i))$ of $f$ corresponds to the fixed point class $p(\Fix(g\alpha \tilde{f}_ig_0))$. Since $g\alpha \tilde{f}_ig_0$ has no fixed points by definition of $g_0$, the index of this fixed point class is $0$. As indices are preserved under homotopies, the index of $p(\Fix(g\alpha \tilde{f}_i))$ is also $0$.
\end{proof}

On the other hand, the results in \cite{charlotte2} give an expression for the number of terms in the second sum in (\ref{eq:decomp-N}):

\begin{lem}\label{lem:R}
For a morphism $\phi':S'\to N$, with $N$ finitely generated, torsion free and nilpotent, and $S'$ a finite index subgroup of $N$, one has \[
R(\phi')=[N:S']|\det(I-(\phi')_*)|_\infty,
\]
where $|a|_\infty$ is defined as $|a|$ if $a\neq 0$, and $\infty$ if $a=0$.
\end{lem}

\begin{proof}
Let $c$ be the nilpotency class of $N$. For $i=1,\ldots,c$, write $N_i=\sqrt[N]{\gamma_i(N)}$ and $S'_i=\sqrt[S']{\gamma_i(S')}$. If $M_i$ denotes the matrix of the morphism $S'_iN_{i+1}/N_{i+1}\to N_i/N_{i+1}$ induced by $\phi'$, it is proven in \cite[Theorem 4.14]{charlotte2} that \[
R(\phi')=[N:S']|\det(I-M)|_\infty,
\]
where \[
M = \begin{bmatrix}
M_1 & 0 & \cdots & 0 \\
0 & M_2 & \cdots & 0 \\
\vdots & \vdots & \ddots & \vdots \\
0 & 0 & \cdots & M_c
\end{bmatrix}.
\]

On the other hand, it is proven in \cite[Theorem 5.4]{charlotte2} that the matrix of $(\phi')_*$ with respect to a well-chosen basis for $\g$ is of the form \[
\begin{bmatrix}
M_1 & 0 & \cdots & 0 \\
* & M_2 & \cdots & 0 \\
\vdots & \vdots & \ddots & \vdots \\
* & * & \cdots & M_c
\end{bmatrix}.
\]
Thus, we can replace $\det(I-M)$ by $\det(I-(\phi')_*)$, which gives the desired formula.
\end{proof}

Lastly, we observe:

\begin{lem}\label{lem:det}
Let $f:\pi\orb G\to D_n(\pi\orb G)$ be an affine $n$-valued map with lift $\tilde{f}:x\mapsto (g_1\varphi_1(x),\ldots,g_n\varphi_n(x))$ and corresponding morphisms $\phi_i':S_i'\to N$. For all $i$ and $\alpha\in \pi$, for all $g\in N$,
\[
\det(I-(\tau_{g\alpha}\phi'_i)_*)=\det(I-(A_{\alpha}\varphi_i)_*).
\]
\end{lem}

\begin{proof}
We may choose a basis $v_{1,1}, v_{1,2}, \ldots, v_{1,k_1},v_{2,1}, \ldots , v_{c,k_c}$ for $\g$ in such a way that the vectors $v_{i,1}, v_{i,2}, \ldots, v_{i,k_i},v_{i+1,1}, \ldots , v_{c,k_c}$ form a basis of $\gamma_i(\g)$, the $i$-th term in the lower central series of the Lie algebra $\g$. With respect to this basis, all morphisms $\g\to\g$ are represented by blocked lower triangular matrices (with blocks of size $k_i \times k_i$ on the diagonal).

Now note that for the morphisms $\tau_{g\alpha},\,\phi_i':S_i'\to N$, and therefore also for their extensions $G\to G$, \[
\tau_{g\alpha}=c_{ga_\alpha} A_\alpha \qquad \text{and} \qquad \phi_i'=c_{g_i}\varphi_i
\]
where $c_g:G\to G$ denotes conjugation with $g\in G$. (The equality for $\phi_i$ follows by evaluating (\ref{eq:psi-aff}) at $x=1$ and $\gamma\in S_i'$.)

For all $g\in G$, the map $c_g: G \to G$ induces the identity on $\gamma_i(G)/\gamma_{i+1}(G)$, so the maps $(c_g)_\ast$ induce the identity on $\gamma_i(\g)/\gamma_{i+1}(\g)$ and the matrix representation of $(c_g)_\ast$ has identity blocks on the diagonal. It follows that the matrices of $(\tau_{g\alpha})_*$ and $(\phi'_i)_*$ have the same blocks on the diagonal as those of $(A_\alpha)_*$ and $(\varphi_i)_*$, respectively; therefore \[
\det(I-(\tau_{g\alpha}\phi'_i)_*)=\det(I-(A_{\alpha}\varphi_i)_*). \qedhere
\]
\end{proof}

Putting all these results together gives the averaging formula:

\begin{thm}\label{mainthm}
The Nielsen number of an affine map $f:\pi\orb G\to D_n(\pi\orb G)$ with lift $\tilde{f}:x\mapsto (g_1\varphi_1(x),\ldots,g_n\varphi_n(x))$ is given by \[
N(f)=\frac{1}{[\pi:N]}\sum_{\bar{\alpha}\in \pi/N}\sum_{i=1}^n |\det(I-(A_\alpha\varphi_i)_*)|.
\]
\end{thm}

\begin{proof}
Recall from section \ref{sec:decomp} that \[
N(f)=\sum_{i=1}^n\frac{1}{[\pi:S'_i]}\sum_{\bar{\alpha}\in \pi/N}\sum_{[g]\in \R[\tau_\alpha\phi'_i]}|\overline{\fix(\tau_{g\alpha}\phi_i)}|\,\varepsilon_{g\alpha,i}.
\]
In case $\det(I-(A_\alpha\varphi_i)_*)\neq 0$, Lemma \ref{lem:fix} implies $\fix(\tau_{g\alpha}\phi_i)=1$, so in particular $|\overline{\fix(\tau_{g\alpha}\phi_i)}|=1$; Corollary \ref{cor:eps1} implies $\varepsilon_{g\alpha,i}=1$ for all $g$; and Lemmas \ref{lem:R} and \ref{lem:det} imply \[
R(\tau_\alpha\phi'_i)=[N:S'_i]|\det(I-(A_\alpha\varphi_i)_*)|.
\]
Thus, in this case, \[
\sum_{[g]\in \R[\tau_\alpha\phi'_i]}|\overline{\fix(\tau_{g\alpha}\phi_i)}|\,\varepsilon_{g\alpha,i}=[N:S'_i]|\det(I-(A_\alpha\varphi_i)_*)|.
\]
In case $\det(I-(A_\alpha\varphi_i)_*)=0$, Corollary \ref{cor:eps0} implies $\varepsilon_{g\alpha,i}=0$ for all $g$, so \[
\sum_{[g]\in \R[\tau_\alpha\phi'_i]}|\overline{\fix(\tau_{g\alpha}\phi_i)}|\,\varepsilon_{g\alpha,i}=0=[N:S'_i]|\det(I-(A_\alpha\varphi_i)_*)|.
\]
It follows that \begin{align*}
N(f)&=\sum_{i=1}^n\frac{1}{[\pi:S'_i]}\sum_{\bar{\alpha}\in \pi/N}[N:S'_i]|\det(I-(A_\alpha\varphi_i)_*)| \\
&=\frac{1}{[\pi:N]}\sum_{\bar{\alpha}\in \pi/N}\sum_{i=1}^n\,|\det(I-(A_\alpha\varphi_i)_*)|. \qedhere
\end{align*}
\end{proof}

\section{Example}

To illustrate the formula, consider the following example. Take $G=\RR^2$ and let $\pi\subseteq \RR^2\rtimes \mathrm{GL}_2(\RR)$ be the group generated by \[
a=\left(\begin{bmatrix}1 \\ 0 \end{bmatrix},\begin{bmatrix} 1 & 0 \\ 0 & 1 \end{bmatrix}\right), \quad b=\left(\begin{bmatrix}0 \\ \frac{1}{2} \end{bmatrix},\begin{bmatrix} -1 & 0 \\ 0 & 1 \end{bmatrix}\right).
\]
Then $\pi\orb \RR^2$ is an infra-nilmanifold isomorphic to the Klein bottle. Consider the map given by \[
\tilde{f}: \RR^2\to F_2(\RR^2,\pi): (t_1,t_2)\mapsto ((0,\textstyle\frac{t_2}{2}),(0,\textstyle\frac{t_2}{2}-\frac{1}{4})).
\]
This map induces a well-defined $2$-valued map $f:\pi\orb \RR^2\to D_2(\pi\orb \RR^2)$ on the Klein bottle, since for all $(t_1,t_2)\in \RR^2$ \begin{align*}
\tilde{f}(a(t_1,t_2))
&=((0,\textstyle\frac{t_2}{2}),(0,\textstyle\frac{t_2}{2}-\frac{1}{4}))=(1,1;1)\tilde{f}(t_1,t_2) \\
\tilde{f}(b(t_1,t_2))
&=(b(0,\textstyle\frac{t_2}{2}-\frac{1}{4}),(0,\textstyle\frac{t_2}{2}))=(b,1;\sigma)\tilde{f}(t_1,t_2)
\end{align*}
with $\sigma\in \S_2$ the permutation that interchanges $1$ and $2$.

Note that $f$ is an affine map, with \[
g_1=\begin{bmatrix}0 \\ 0 \end{bmatrix}, \quad g_2=\begin{bmatrix}0 \\ -\frac{1}{4} \end{bmatrix}, \quad \varphi_1=\varphi_2=\begin{bmatrix} 0 & 0 \\ 0 & \frac{1}{2} \end{bmatrix}.
\]
To compute the Nielsen number of $f$, we consider the nilmanifold $Z\orb \RR^2$ where $Z=\pi\cap \RR^2=\langle a,b^2 \rangle \cong \ZZ^2$, i.e. the $2$-torus. The group $\pi/Z$ has two elements, represented by \[
\left(\begin{bmatrix}0 \\ 0 \end{bmatrix},\begin{bmatrix} 1 & 0 \\ 0 & 1 \end{bmatrix}\right) \quad \text{and} \quad \left(\begin{bmatrix}0 \\ 0 \end{bmatrix},\begin{bmatrix} -1 & 0 \\ 0 & 1 \end{bmatrix}\right).
\]
Thus, by Theorem \ref{mainthm} \begin{align*}
N(f) &= \frac{1}{2} \sum_{\bar{\alpha}\in \pi/Z} \sum_{i=1}^2 \left|\,\det\left(I-A_\alpha\begin{bmatrix} 0 & 0 \\ 0 & \frac{1}{2} \end{bmatrix}\right)\right| \\
&= \sum_{\bar{\alpha}\in \pi/Z} \left|\,\det\left(I-A_\alpha\begin{bmatrix} 0 & 0 \\ 0 & \frac{1}{2} \end{bmatrix}\right)\right| \\
&=\left|\,\det\left(I-\begin{bmatrix} 1 & 0 \\ 0 & 1 \end{bmatrix}\begin{bmatrix} 0 & 0 \\ 0 & \frac{1}{2} \end{bmatrix}\right)\right|+\left|\,\det\left(I-\begin{bmatrix} -1 & 0 \\ 0 & 1 \end{bmatrix}\begin{bmatrix} 0 & 0 \\ 0 & \frac{1}{2} \end{bmatrix}\right)\right| \\
&=\left|\,\det\begin{bmatrix} 1 & 0 \\ 0 & \frac{1}{2} \end{bmatrix}\right|+\left|\,\det\begin{bmatrix} 1 & 0 \\ 0 & \frac{1}{2} \end{bmatrix}\right|=1.
\end{align*}

A direct calculation shows that the map $f$ has one isolated fixed point: indeed, for all $k,\ell\in \ZZ$ we have \begin{align*}
(t_1,t_2)=a^kb^\ell(0,\textstyle\frac{t_2}{2}) &\iff (t_1,t_2)=(k,\textstyle\frac{t_2}{2}+\frac{\ell}{2}) \\ &\iff (t_1,t_2)=(k,\ell)
\end{align*}
and
\begin{align*}
(t_1,t_2)=a^kb^\ell(0,\textstyle\frac{t_2}{2}-\frac{1}{4}) &\iff (t_1,t_2)=(k,\textstyle\frac{t_2}{2}-\frac{1}{4}+\frac{\ell}{2}) \\ &\iff (t_1,t_2)=(k,\textstyle-\frac{1}{2}+\ell);
\end{align*}
and all these points $(k,\ell)$ and $(k,-\frac{1}{2}+\ell)$ in $\RR^2$ project to the same point on the Klein bottle. Thus, the minimal number of fixed points is actually attained in this case.

\begin{rmk}
This map $f$ is an example of a map on the Klein bottle that does not lift to the torus $Z\orb \RR^2$. In fact, there is no nilmanifold (torus) $N\orb \RR^2$ that finitely covers $\pi\orb \RR^2$ so that $f$ lifts to $N\orb \RR^2$: in order for such a lift to exist, the induced morphism of covering groups $f_*=(\phi_1,\phi_2;\sigma):\pi\to \pi^2 \rtimes \S_2$ defined in section \ref{sec:intro} must satisfy $f_*(N) \subseteq N^2\rtimes \S_2$. However, note that $f_*(b^2)=(b,b;1)$. If $N\orb \RR^2$ is a nilmanifold that finitely covers $\pi\orb \RR^2$, then $N$ must be a finite index subgroup of $Z$. Hence there is a smallest $\ell\in\ZZ_{>0}$ so that $(b^2)^\ell\in N$. But then $f_*((b^2)^\ell)=(b^\ell,b^\ell;1)$ is not an element of $N^2\rtimes \S_2$, by minimality of $\ell$.

This illustrates the fact that a classical averaging formula in terms of lifts as in \cite{kimleelee} is not feasible in the $n$-valued case.
\end{rmk}

\end{document}